\documentclass[12pt]{amsart}
\usepackage{amssymb,latexsym, amsmath, amscd, array, graphicx}

\def\cat{{\mathop\mathrm{cat}}}

\def\co{\colon}

\def\id{\mathop\mathrm{id}}

\usepackage{tikz-cd}

\usepackage{graphicx}
\newtheorem{theorem}{Theorem}
\newtheorem{lemma}[theorem]{Lemma}
\newtheorem{corollary}[theorem]{Corollary}

\newtheorem{problem}[theorem]{Problem}
\newtheorem{prop}[theorem]{Proposition}

\theoremstyle{definition}
\newtheorem{definition}[theorem]{Definition}

\newtheorem{remark}[theorem]{Remark}

\title[]{The Lusternik-Schnirelmann category of connected sum}
\author{Alexander Dranishnikov, Rustam Sadykov}

\begin{document}
\maketitle

\begin{abstract}
We use the Berstein-Hilton invariant to prove the formula $\cat(M_1\sharp M_2)=\max\{\cat M_1, \cat M_2\}$ for the Lustrnik-Schnirel\-mann category of the connected sum
of closed  manifolds $M_1$ and $M_2$. 
\end{abstract}

\section{Introduction}

The \emph{Lusternik-Schnirelmann category} (the LS-category)  $\cat X$ of a topological space $X$ is the least number $n$ such that there is a covering  $\{U_i\}$ of $X$ by $n+1$ open sets $U_i$ contractible in $X$ to a point. The LS-category  gives a lower bound for the number of critical points of a function on a manifold~\cite{LS}, \cite{Co}. In the late 20s early 30th it was used 
by Lusternik and Schnirelmann to solve the Poincare problem on the existence of three closed geodesics on the sphere~\cite{LS29}. Since then
the LS-category brought many applications to different areas of mathematics. In particular, it was used in establishing the Arnold conjecture for symplectic manifolds~\cite{Ru}.  

Computation of the LS-category of manifolds presents a great challenge. The behavior of the LS-category under two basic operation on manifolds, the product and the connected sum, 
for a long time was a complete mystery. For the product there was a long standing natural conjecture known as the Ganea Conjecture: $\cat(M\times S^n)=\cat M+1$.
We note that the inequality $\cat(M\times N)\le\cat M+\cat N$ holds for all spaces~\cite{Co}. Since $\cat S^n=1$, the Ganea Conjecture states that the LS-category increases if one cross a manifold with the $n$-sphere. As for the connected sum, the natural conjecture was that
 $$\cat(M\sharp N)=\max\{\cat M,\cat N\}.$$

In the early 00s Norio Iwase disproved the Ganea Conjecture~\cite{Iw}. His counterexample left little hope for a reasonable formula for the LS-category of the product.

Contrary to the product, the connected sum conjecture was almost proven lately in the sequel of works~\cite{Ne},\cite{Dr14b},\cite{DS}. 
In~\cite{Dr14b} the inequality $\cat(M\sharp N)\le\max\{\cat M,\cat N\}$ was proven and then
in~\cite{DS} the connected sum formula was established for all orientable manifolds.
Unfortunately the method of the proof of the connected sum formula  in~\cite{DS} does not work when both manifolds are non-orientable.

Surprisingly, the Iwase's counterexample led us to a new proof of the connected sum conjecture 
which works in full generality. The main feature of our proof (as well as of Iwase's counterexample) is the use of the Berstein-Hilton invariant.

The main result of this paper is the following:

\begin{theorem}\label{th:1} There is an equality
$\cat(M_1\sharp M_2)=\max\{\cat M_1, \cat M_2\}$ for closed  manifolds $M_1$ and $M_2$. 
\end{theorem}

\section{The Berstein-Hilton invariant}

\subsection{Homotopically trivial fibrations}
We will assume that all fibrations $p\co E\to B$ under consideration are pointed, i.e., there is a distinguished point in $E$ which maps to a distinguished point in $B$. Furthermore, we choose a distinguished point in the fiber $F$ so that the inclusion $F\to E$ is pointed. Then  a pointed map $f\co X\to E$ for which $p\circ f$ is constant factors through a pointed map $f\co X\to F$. We will assume that the base, the total space and the fiber of $p$ are path connected.  

Unless otherwise stated, we assume that sections are pointed. In particular, given a fibration $p\co E\to A\vee B$, any two sections $\sigma_A$ and $\sigma_B$ of $p|A$ and $p|B$ define a section $\sigma$ of $p$. We also note that given a pointed section $\sigma$ of $p$, its restrictions $\sigma_A$  and $\sigma_B$ over $A$ and $B$ respectively are pointed sections.

Suppose that the inclusion of the fiber $F$ into the total space $E$ of a fibration $p:E\to B$ induces an injective homomorphism $\pi_i(F)\to \pi_i (E)$ for all $i$, and for all $i$ there exists a retraction $r_p\co \pi_i(E)\to \pi_i (F)$. Then the pair $(p, r_p)$ is said to be a \emph{homotopically trivial} fibration. Slightly abusing notation we will say that $p$ is a homotopically trivial fibration. We note that if $p$ is a homotopically trivial fibration, then there is a chosen isomorphism $\pi_i (E)\approx \pi_i( F)\oplus \pi_i (B)$. 

 \begin{lemma}\label{l:7a}
Any fibration $p:E\to B$ that admits a section is homotopically trivial.  
\end{lemma}
\begin{proof}  Since $p$ admits a section, the homotopy exact sequence for $p$ splits:
\[
    0  \to \pi_i(F)\to \pi_i(E)\stackrel{p_*}\to \pi_i (B) \to 0
\]
where $F$ denotes the fiber of $p$.
Let  the homomorphism $\sigma_*:\pi_i B\to \pi_i(E)$ be the one induced by a section $\sigma:B\to E$. Then, identifying the group $\pi_i(F)$ with a subgroup of $\pi_i(E)$, we define the retraction $r_p:\pi_i(E)\to \pi_i(F)$ by the formula $r_p(\gamma)= \gamma-(\sigma_*p_*)(\gamma)$.  
\end{proof}

Let $p:E\to B$ be a pointed fibration with fiber $F$. Then the loop space functor defines a pointed fibration $\Omega p:\Omega E\to \Omega B$ with fiber $\Omega F$.
The following Lemma follows from Lemma~\ref{l:7a}. 
\begin{lemma}\label{loop}
Supose that the loop fibration $\Omega p:\Omega E\to \Omega B$ admits a section. Then the fibration $p:E\to B$ is homotopy trivial
\end{lemma}
 
For $i=1,2$, let $p_i\co E_i\to B_i$ be fibrations with fibers $F_i$ respectively. A map $f\co p_1\to p_2$ of fibrations is a commutative diagram:
\[
  \begin{tikzcd}
  F_1 \arrow[r] \arrow[d, "f_F"]& E_1 \arrow[d, "f_E"]\arrow[r]  & B_1 \arrow[d, "f_B"] \\
   F_2 \arrow[r] &  E_2 \arrow[r] &  B_2.  
  \end{tikzcd}
\]
The vertical maps in the diagram are called components of $f$; these are denoted by $f_F, f_E$ and $f_B$ respectively.  We say that a map $f\co p_1\to p_2$ of fibrations is a \emph{retraction}, if there is a map $g\co p_2\to p_1$ of fibrations such that $f\circ g=\id$. A map $f\co p_1\to p_2$ is a \emph{homotopy equivalence} if there is a map 
$f^{-1}\co p_2\to p_1$ of fibrations such that $f\circ f^{-1}$ and $f^{-1}\circ f$ are homotopic to identity maps of fibrations. We note that if $f$ is a homotopy equivalence, then its components are homotopy equivalences. 

We say that $f$ is an \emph{embedding}, if its components are embeddings. In this case we also say that $p_1$ is a \emph{subfibration} of $p_2$. A section $\sigma_1$ of the subfibration $p_1$ defines a section $\sigma_2$ of $p_2$ over $f_B(B_1)\subset B_2$ by $\sigma_2=f_E\circ \sigma_1\circ f_B^{-1}$. An extension of $\sigma_2$ to a section of $p_2$ is said to be an \emph{extension} of $\sigma_1$. 

A {\em morphism} $f\co p\to p'$ of homotopically trivial fibrations is a map of fibrations for which the following diagram is commutative for all $i$:
\[ ({\rm I})\ \ \ 
  \begin{tikzcd}
   \pi_i(E)\arrow[r, "r_p"] \arrow[d, "(f_E)_*"]& \pi_i(F) \arrow[d, "(f_F)_*"]\\
   \pi_i(E')\arrow[r, "r_{p'}"] &   \pi_i(F'). 
  \end{tikzcd}
\] 
We use the concepts of embedding, subfibration, homotopy equivalence, and retract in the category of homotopically trivial fibrations.
\subsection{Definition of the Berstein-Hilton invariant}
Let $(p,r_p)\co E\to B$ be a homotopically trivial fibration with fiber $F$. Suppose that $p$ admits a section. Let $\alpha\in \pi_i(B)$.
The set $$H(\alpha)=\{r_p\circ\sigma_*(\alpha)\in\pi_i(F)\}$$ where $\sigma$ ranges over all sections of the fibration $p$ is the \emph{Berstein-Hilton invariant} of $\alpha$.
When we want to emphasize  the underlying homotopically trivial fibration $(p,r_p)$ in the definition of the Berstein-Hilton invariant we use the notation $H(\alpha,r_p)$ instead of $H(\alpha)$.
   
 \begin{remark}  The above definition is a generalization of the Berstein-Hilton invariant defined in  \cite{BH}. 
 \end{remark}  
  
  \section{Properties of the Berstein-Hilton invariant.}

\subsection{Subfibrations.} 
 

The following theorem slightly generalizes Theorem 6.19 from~\cite{CLOT}.

 \begin{theorem}\label{section} 
\label{th:0} Let $(p',r_{p'})\co E'\to B'$ be a homotopically trivial subfibration of a homotopically trivial fibration $(p,r_p)\co E\to B$. Suppose that  $B=B'\cup_\phi D$ is obtained from $B'$ by attaching an $(n+1)$-cell $D$ along a map representing $\alpha\in \pi_n(B')$. Suppose that $p'$ admits a section, and $H(\alpha,r_{p'})$ contains 0. Then  $p$ admits a section. 
\end{theorem}
\begin{proof}
Let $\sigma':B'\to E'$ be a section with $r_{p'}\circ \sigma'_*(\alpha)=0$. We may extend $\sigma'$ to the complement 
$D\setminus Int(D_0)$
of 
the interior of an $(n+1)$-ball $D_0\subset Int(D)$.
The section $\sigma'$ can be further extended to a section over $B$ if and only if  $j_*[\sigma'|_{\partial D_0}]$ is trivial in $\pi_n(p^{-1}(D_0))$
where $j:p^{-1}(\partial D_0)\to p^{-1}(D_0)$ is the inclusion. 
In the commutative diagram generated by the inclusions and the retractions
$$
\begin{CD}
\pi_n(p^{-1}(\partial D_0) )@. \pi_n(E') @>r_{p'}>> \pi_n(F')\\
@ Vj_*VV @Vi_*VV @ Vi'_*VV\\
\pi_n(p^{-1}(D_0)) @>j'_*>>\pi_n(E) @>r_p>>\pi_n(F)\\
\end{CD}
$$
the composition $r_pj'_*$ is an isomorphism. 
Note that 
$j'_*j_*[\sigma'|_{\partial D_0}]=i_*(\sigma'_*(\alpha))$. Hence  $r_pj'_*j_*[\sigma'|_{\partial D_0}]=i'_*r_{p'}\sigma'_*(\alpha)=0$.
Therefore, the homotopy class $j_*[\sigma'|_{\partial D_0}]$ is trivial in $\pi_n(p^{-1}(D_0))$.
\end{proof}

The following proposition easily follows from the definition.
 
 \begin{prop}\label{l:1} 
Let $p':E'\to B'$ be a homotopically trivial subfibration of  homotopically trivial fibration $p:E\to B$. Then
for any section $\sigma'$ of  $p'$  and for each section $\sigma$ of $p$ extending $\sigma'$, there is a commutative  diagram:
 \[
 \begin{tikzcd}
 \pi_i (B') \arrow[r, "r_{p'}\sigma'_*"]  \arrow[d] & \pi_i(F' )\arrow[d] \\
 \pi_i( B )\arrow[r, "r_p\sigma_*"] & \pi_i( F ) 
 \end{tikzcd}
 \]
 where the two vertical homomorphisms are induced by inclusions.
 \end{prop}


\subsection{Retracts} 
 
 Let $p\co E\to B$ be a fibration with fiber $F$, and $p:E'\to B'$  a retract of $p$ with an inclusion $I\co p'\to p$ and a retraction $\rho\co p\to p'$. Then a section $\sigma$ of $p$ defines a section $\sigma'$ of $p'$ by $\sigma'(b)=\rho_E\circ \sigma\circ I_B(b)$. We say that $\sigma'$ is obtained from $\sigma$ by \emph{projection} to $p'$.

 \begin{prop}\label{l:2} 
Let $\rho:p\to p'$ be  a retraction  in the category of  homotopically trivial fibrations. Suppose that $p$ admits a section $\sigma$. Let $\sigma'$ be the section of $p'$ obtained from $\sigma$ by projection. Then for each $i$ there is a commutative diagram:
 \[
 \begin{tikzcd}
 \pi_i(B') \arrow[r, "r_{p'}\sigma'_*"]  & \pi_i( F') \\
 \pi_i(B) \arrow[r, "r_p\sigma_*"]\arrow[u] & \pi_i( F) \arrow[u],  
 \end{tikzcd}
 \]
 where 
 the two vertical homomorphisms are induced by  $\rho_B$ and $\rho_F$ respectively.
 \end{prop}
 \begin{proof} The commutative diagram 
 \[
 \begin{tikzcd}
 \pi_i( B') \arrow[r, "\sigma'_*"] & \pi_i( E')\arrow[r,"r_{p'}"]&  \pi_i( F') \\
 \pi_i( B) \arrow[r, "\sigma_*"]  \arrow[u, "(\rho_B)_*"]   & \pi_i( E) \arrow[u, "(\rho_E)_*"]\arrow[r,"r_p"]   & \pi_i(F) \arrow[u,"(\rho_F)_*"],  
 \end{tikzcd}
 \]
 implies Proposition 6. 
 \end{proof}

\subsection{Wedge sum $A\vee B$.} 
 For a homotopically trivial fibration $p:E\to X$ with a section $\sigma$ we denote by $H_\sigma$ the homomorphism $r_p\circ\sigma_*:\pi_i(X)\to\pi_i(F)$.

 Let $p_A\co E_A\to A, p_B\co E_B\to B$ and $p\co E\to A\vee B$ be fibrations with path connected fibers $F_A, F_B$ and $F$ respectively. Suppose that $p_A$ and $p_B$ are 
homotopically trivial subfibrations of $p$.
Since all sections under consideration are pointed,  $\sigma_A$ and $\sigma_B$ define a section $\sigma$ of $p$, called the \emph{union} of the sections $\sigma_A$ and $\sigma_B$. 
 
By Proposition~\ref{l:1}, there is a commutative diagram
\begin{equation}\label{eq:2}
     \begin{tikzcd}
       \pi_n A \arrow[r]\arrow[d, "H_{\sigma_A}"] & \pi_n (A\vee B) \arrow[d, "H_{\sigma}"] & \pi_n(B) \arrow[l] \arrow[d, "H_{\sigma_B}"]\\
       \pi_n (F_A)\arrow[r,"i_A"]        & \pi_n (F)   & \pi_n (F_B) \arrow["i_B",l].
     \end{tikzcd}
\end{equation}

 Since $\pi_nA$ and $\pi_nB$ are direct summands of $\pi_n(A\vee B)$, 
 we may identify elements $\alpha\in \pi_n(A)$ and $\beta\in \pi_n( B)$ with the corresponding elements in $\pi_n(A\vee B)$. 

 \begin{lemma}\label{th:3} Let $p_A, p_B$ and $p$ be homotopically trivial fibrations as above. Let $\sigma_A$, $\sigma_B$ and $\sigma$ be sections of respectively $p_A, p_B$ and $p$ such that $\sigma$ is a union of $\sigma_A$ and $\sigma_B$. 
 Then for any elements $\alpha\in \pi_n(A)$ and $\beta\in \pi_n(B)$, $n>1$, we have $$H_\sigma(\alpha+\beta)=i_AH_{\sigma_A}(\alpha)+i_BH_{\sigma_B}(\beta).$$
 \end{lemma}
 \begin{proof}   By the commutativity of the diagram (\ref{eq:2}), the image of $H_{\sigma_A}(\alpha)$ in $\pi_n(F)$ is $H_\sigma(\alpha)$. Similarly, the image of $H_{\sigma_B}(\beta)$ in $\pi_n(F)$ is $H_\sigma(\beta)$. Since $H_\sigma$ is a homomorphism, the statement of Lemma~\ref{th:3} immediately follows. 
  \end{proof}

 \begin{lemma}\label{th:4}  Let $p_A:E_A\to A$ and $p_B:E_B\to B$ be homotopically trivial subfibrations of a homotopically trivial fibration $p:E\to A\vee B$
such that collapsing maps $q_A:A\vee B\to A$ and $q_B:A\vee B\to B$ extend to retractions $\rho_A:p\to p_A$, $\rho_B:p\to p_B$
in the category of homotopically trivial fibrations. Suppose that $p$ admits a section $\sigma$.  Let $\sigma_A$ and $\sigma_B$ be the projections of $\sigma$ to $p_A$ and $p_B$ respectively.
 Suppose that 
 $H_\sigma(\alpha+\beta)=0$ for $\alpha\in\pi_n(A)$ and $\beta\in\pi_n(B)$. Then  $H_{\sigma_A}(\alpha)=0$ and $H_{\sigma_B}(\beta)=0$.
 \end{lemma}
 \begin{proof} Let $\sigma'_A$ denote the section of $p_A$ obtained from the section $\sigma$ of  $p$  by projection. Since $p_A$ is a retract of $p$,   Proposition~\ref{l:2} 
implies $H_{\sigma_A}\circ(\rho_A)_*=(\rho_{F_A})_*\circ H_\sigma$. Hence, $H_{\sigma_A}(\alpha)=H_{\sigma_A}\circ(\rho_A)_*(\alpha+\beta)=(\rho_{F_A})_*\circ H_\sigma(\alpha+\beta)=0$. Similarly, $H_{\sigma_B}(\beta)=0$.
\end{proof}

 \section{Lusternik-Schnirelmann category}  

\subsection{Ganea fibrations} We recall that an element of an iterated join $X_0*X_1*\cdots*X_n$ of topological spaces is a formal linear combination $t_0x_0+\cdots +t_nx_n$ of points $x_i\in X_i$ with $\sum t_i=1$, $t_i\ge 0$, in which all terms of the form $0x_i$ are dropped. Given fibrations $f_i\co X_i\to Y$ for $i=0, ..., n$, the fiberwise join of spaces $X_0, ..., X_n$ is defined to be the space
\[
    X_0*_Y\cdots *_YX_n=\{\ t_0x_0+\cdots +t_nx_n\in X_0*\cdots *X_n\ |\ f_0(x_0)=\cdots =f_n(x_n)\ \}.
\]
The fiberwise join of fibrations $f_0, ..., f_n$ is the fibration 
\[
    f_0*_Y*\cdots *_Yf_n\co X_0*_YX_1*_Y\cdots *_YX_n \longrightarrow Y
\]
defined by taking a point $t_0x_0+\cdots +t_nx_n$ to $f_i(x_i)$ for any $i$. As the name `fiberwise join' suggests, the fiber of the fiberwise join of fibrations is given by the join of fibers of fibrations. 

When $X_i=X$ and $f_i=f:X\to Y$ for all $i$  the fiberwise join of spaces is denoted by $*^{n+1}_YX$ and the fiberwise join of fibrations is denoted by $*_Y^{n+1}f$. 
For a path connected topological space $X$, we turn an inclusion of a point $*\to X$ into a fibration $p_0^X:G^0X\to X$. The $n$-th Ganea space of $X$ is defined to be the space $G_nX=*_X^{n+1}G_0X$, while the $n$-th Ganea fibration $p_n^X:G_nX\to X$ is the fiberwise join of fibrations $p_0^X:G_nX\to X$.
Thus, the fiber of $p_n^X$ is iterated joint product $\ast^{n+1}\Omega X$ of the loop space of $X$.

 The following theorem is called Ganea's characterization of the LS-category. It was proven first in~\cite{Sch} in a greater generality.
\begin{theorem}  
 Let $X$ be a connected  CW-complex. Then $\cat(X)\le n$ if and only if the fibration $p_n^X:G_nX\to X$ admits a section. 
\end{theorem}
\begin{corollary}
Suppose that $\cat X\le n$. Then the Ganea fibration $\pi_X^n$ is homotopically trivial. 
\end{corollary}
Surprisingly, for the Ganea fibrations much stronger statement holds true~\cite{BH}.
\begin{theorem}\label{1}
For any $n>1$, the $n$-th Ganea fibration $p_n^X:G_nX\to X$ is homotopically trivial.

Moreover, any continuous map $f:X\to Y$ induces a morphism $\bar f:p_n^X\to p^Y_n$ of homotopically trivial Ganea fibrations.
\end{theorem}
The proof of the theorem can be derived from~\cite[Proposition 2.8]{BH}. Since Proposition 2.8 in~\cite{BH} is stated and proven in a different language,
we give a sketch of the proof of Theorem~\ref{1}. 

\begin{proof}
The main fact behind the theorem is the equality $\cat S^i=1$ for all $i>0$. For any $n>1$ and $i\ge 1$, we define a natural homomorphism $s^i:\pi_i(X)\to\pi_i(G_nX)$ as follows.
Fix a section $\sigma:S^i\to G_nS^i$. Since $\cat S^i=1$, we may assume that $\sigma$ lands in $G_1S^i\subset G_nS^i$.
 Let $\alpha\in\pi_i(X)$. We define $s^i(\alpha)=(\bar\phi)_*\sigma_*(a)$ where $\phi:S^i\to X$ is a map representing $\alpha$, the map
$\bar\phi:G_nS^i\to G_nX$ is the induced map of the $n$-th Ganea's fibrations, and $a$ is a fixed generator of $\pi_i(S^i)$: 
\[
\begin{CD}
G_nS^i @>\bar\phi>> G_nX\\
@A\sigma AA @Vp^X_nVV\\
S^i @>\phi>> X.\\
\end{CD}
\]
One can prove that the map $s^i$ is well-defined. 
The condition $n>1$ allows us to show that $s^i$ is a homomorphism. Clearly, $s^i$ is a section of the homomorphism $(p^X_n)_*:\pi_i(G_nX)\to\pi_i(X)$.

Then we define the retraction $r_{p^X_n}:\pi_i(G_nX)\to\pi_i(F)$ by the formula 
$$
r_{p^X_n}(\beta)=\beta-s^i(p^X_n)_*(\beta).
$$
\end{proof}

REMARK. More traditional proof of Theorem~\ref{1} is based on  Lemma~\ref{loop} and the fact that the loop fibration $\Omega p^X_n:\Omega G_nX\to\Omega X$ admits a section
(see~\cite{CLOT}).

 \begin{lemma}\label{less}
Let $\alpha\in\pi_i(X)$. 
If $H(\alpha)$ contais 0 for the $\ell$-th Ganea fibration $p^X_\ell$ or $\cat X\le\ell$, then $H(\alpha)$ contains 0 for the $k$-th Ganea fibration $p^X_k$ for all $k>\ell$.
\end{lemma}
\begin{proof}
In the first case this follows from the fact that $p^X_\ell$ is a homotopically trivial subfibration of a homotopically trivial fibration $ p^X_k$.
In the second case this follows from the fact that for $k>\ell$ the inclusion of fibers $F_\ell\to F_k$ of these fibrations is null-homotopic.
\end{proof}

Let $A$ and $B$ be two connected CW-complexes. Denote by $p_A, p_B$ and $p$ the $n$-th Ganea fibrations, $n>1$, over $A, B$ and $A\vee B$ respectively. Then $p_A$ and $p_B$ are retractions of $p$ in the category of homotopically trivial fibrations.  We will denote the fibers of $p_A, p_B$ and $p$ by $F_A, F_B$ and $F$ respectively.  We assume that $p$ admits a section.

\begin{theorem}\label{th:5a}  Let $p$ be as above. Let $\alpha\in \pi_i(A)$ and $\beta\in \pi_i(B)$ be two elements of homotopy groups. Then $H(\alpha+\beta)$ contains zero if and only if  both $H(\alpha)$ and $H(\beta)$ contain zero. 
\end{theorem} 
\begin{proof}  
Suppose that $H(\alpha+\beta)$ contains zero. Then there is a section $\sigma$ of $p$ such that $H_\sigma(\alpha+\beta)=0$. By Lemma~\ref{th:4}, 
$H_{\sigma_A}(\alpha)=0$ and $H_{\sigma_B}(\beta)=0$ for sections $\sigma_A$ and $\sigma_B$ obtained from $\sigma$ by projections. Thus,  both $H(\alpha)$ and $H(\beta)$ contain zero.

Suppose that both $H(\alpha)$ and $H(\beta)$ contain zero. Then there are sections $\sigma_A$ and $\sigma_B$ of $p_A$ and $p_B$ such that $H_{\sigma_A}(\alpha)=0$
and $H_{\sigma_B}(\beta)=0$. The union of sections $\sigma_A$ and $\sigma_B$ is a section $\sigma$ of $p$. By Lemma~\ref{th:3},

\[
  H_\sigma(\alpha+\beta)=i_A H_{\sigma_A}(\alpha) + i_B H_{\sigma_B}(\beta)=0. 
\]
\end{proof}

\section{Iwase manifolds}

We recall that a map $f:X\to Y$ is called an \emph{$n$-equivalence} if it induces isomorphisms of homotopy groups $f_*:\pi_i(X)\to\pi_i(Y)$ for $i<n$ and an epimorphism for $i=n$.
We use the following proposition (a proof can be found in~\cite[Proposition 3.3]{DS} or in~\cite[Proposition 5.7]{DKR}).
\begin{prop}\label{old}
 Let $f\co X\to Y$ be an $s$-equivalence of pointed $r$-connected CW-complexes with $r\ge 0$. Then the induced map $f_k':F_k^X\to F_kY$ of the fibers of the $k$-th Ganea spaces is a $(k(r+1)+s-1)$-equivalence. 
\end{prop}

Given a closed connected manifold $M$, its once punctured version is homotopy equivalent to the complement $M^\bullet$ in $M$ to an open disc.

The following theorem was proven by Stanley~\cite[Theorem 3.5]{St}.
\begin{theorem}\label{Stanley}
 Let $X$ be a connected CW-complex with $\cat X=k>0$. Then $\cat X^{(s)}\le k$ for any $s$, where  $X^{(s)}$ is the $s$-skeleton of $X$.
\end{theorem}
\begin{proof}
Note that the inclusion $j:X^{(s)}\to X$ is an $s$-equivalence. By Proposition~\ref{old} the induced map $\ast^{k+1}j:\ast^{k+1}\Omega X^{(s)}\to \ast^{k+1}\Omega X$
of the fibers of the $k$-th Ganea fibrations  is a $(k+s-1)$-equivalence.
Since for $k>0$, $s\le k+s-1$, this implies that a section of $p_k^X$ defines a section of $p_k^{X^{(s)}}$.
\end{proof}
\begin{corollary}
For any manifold $M$, $\cat M^\bullet\le\cat M$.
\end{corollary}
Clearly, $\cat M\le\cat M^\bullet+1$.

 Iwase~\cite{I},\cite{Iw} constructed examples of peculiar manifolds $M$ with the property that 
$\cat M=\cat M^\bullet$. 
\begin{definition}
We call a path connected manifold $M$ an \emph{Iwase} manifold if $\cat M=\cat M^\bullet$. 
\end{definition}

Thus, if $M$ is not an Iwase manifold, then $\cat M = \cat M^\bullet +1$.

The following theorem can be derived from results of Berstein-Hilton~\cite{BH}, Stanley~\cite{St}.

\begin{theorem}\label{th:12}Suppose that $\cat M^\bullet= n>1$, and $\dim M=m >2$.  Then $\cat M = n$ if and only if the Berstein-Hilton invariant $H(\alpha)$ for the $n$-th 
Ganea's  fibration $p_n:G_nM^\bullet\to M^\bullet$ and $\alpha=[\partial M^\bullet]\in\pi_{m-1}(M^\bullet)$ contains zero. 
\end{theorem}
\begin{proof}  
If $H(\alpha)$ contains zero, Theorem~\ref{section} applied to the embedding of  the Ganea fibrations $p_n^{M^\bullet}\to p_n^{M}$ implies that $\cat M\le n$.
Then by Theorem~\ref{Stanley}, $\cat M=n$.

Suppose that $\cat M=n$. Then a section $\sigma: M\to G_nM$ defines a section $\sigma': M^\bullet\to G_nM^\bullet$ such that the 
corresponding diagram commutes. This follows from the fact that the inclusion $j:M^\bullet\to M$ is $(m-1)$-connected and in view of Proposition~\ref{old}, the induced map $ j':F'\to F$
of the fibers of the $n$-th Ganea fibrations
is a $(n+m-2)$-equivalence. In view of the formula $(j')_*H_\sigma'(\alpha)=H_\sigma(j_*(\alpha))$ we obtain $(j')_*H_{\sigma'}(\alpha)=0$. Since $j'$ is a $(n+m-2)$-equivalence,  for $n>1$ it induces an isomorphism of $(m-1)$-dimensional homotopy groups.
Therefore, $H_{\sigma'}(\alpha)=0$. This means that $H(\alpha)$ contains 0.
\end{proof}

\begin{theorem}\label{sum}
A connected sum of Iwase manifolds is an Iwase manifold.
\end{theorem}
\begin{proof}
Let $M$ and $N$ be Iwase $n$-manifolds. Let $\cat(M\sharp N)^\bullet=k$. Then $\alpha=[\partial(M\sharp N)^\bullet]=[\partial M^\bullet]+[\partial N^\bullet]\in\pi_n(M^\bullet\vee N^\bullet)$.
Since $\cat((M\sharp N)^\bullet)=\max\{\cat M^\bullet,\cat N^\bullet\}$ we may assume that $\cat M^\bullet=k$ and $\cat N^\bullet=\ell\le k$. Then by Theorem~\ref{th:12},
$H([\partial M^\bullet])$ contains 0 for the $k$-th Ganea fibration and $H([\partial N^\bullet])$ contains 0 for the $\ell$-th Ganea fibration. By Lemma~\ref{less},
$H([\partial N^\bullet])$ contains 0 for the $k$-th Ganea fibration. By Theorem~\ref{th:5a}, $H(\alpha)$ contains 0. Then by Theorem~\ref{th:12}$,  M\sharp N$ is an Iwase manifold.
\end{proof}

\begin{prop}\label{th:1} Suppose that $M\sharp N$ is an Iwase manifold. If $\cat M^\bullet \ge \cat N^\bullet$, then $M$ is an Iwase manifold. 
\end{prop}
\begin{proof} Since $M\sharp N$ is an Iwase manifold, its dimension is $>2$. 
Note that $(M\sharp N)^\bullet=M^\bullet\vee N^\bullet$.
Suppose that $\cat (M^\bullet\vee N^\bullet) = n$. Since $\cat M^\bullet \ge\cat N^\bullet$, we deduce that $\cat M^\bullet = n$.  
 Since $M\sharp N$ is an Iwase manifold, by Theorem~\ref{th:12} the Berstein-Hilton invariant $H([\partial(M\sharp N)^\bullet])$ 
for the $n$-th Ganea fibration contains zero. Consequently, by Theorem~\ref{th:5a}, both $H([\partial M^\bullet])$ and $H([\partial N^\bullet])$ contain zero. Thus, $\cat M\le n$, and therefore  $\cat M=\cat M^\bullet$. 
\end{proof}

\begin{prop}\label{th:2} 
Let $M$ and $N$ be path connected closed manifolds of dimension equal dimension. 
Suppose that $M\sharp N$ is not an Iwase manifold. If $\cat M^\bullet > \cat N^\bullet $,  then $M$ is not an Iwase manifold. If $\cat M^\bullet =\cat N^\bullet$, then either $M$ or $N$ is not an Iwase manifold.  
\end{prop}
\begin{proof} Since there are no Iwase manifolds of dimension $\le 2$, we may assume that the dimension of $M$ and $N$ is $>2$. 

Suppose that $\cat (M^\bullet \vee N^\bullet)= k.$ 
Theorem~\ref{th:12} implies that $H([\partial(M\sharp N)^\bullet])$ for the $k$-th Ganea fibration does not contain 0. Then, by Theorem~\ref{th:5a},  either
$H([\partial M^\bullet])$ or $H([\partial N^\bullet])$ does not contain 0. 

If $\cat M^\bullet =\cat N^\bullet$ and $H([\partial M^\bullet])$ does not contain 0, then $\cat M^\bullet=k$ and, by Theorem~\ref{th:12}, $M$ is not an Iwase manifold.

If $\cat M^\bullet >\cat N^\bullet=\ell$, then, by Lemma~\ref{less}, $H([\partial N^\bullet])$ contains 0 with respect to the $k$-th Ganea fibration. Then 
$H([\partial M^\bullet])$ does not contain 0. By Theorem~\ref{th:12}, $M$ is not an Iwase manifold.
\end{proof}

Theorem~\ref{sum}, Proposition~\ref{th:1}, and Proposition~\ref{th:2} imply the following.
\begin{theorem}\label{basis}
Let $M$ and $N$ be path connected closed manifolds of the same dimension. 

(1) If $\cat M^\bullet > \cat N^\bullet$, then $M\sharp N$ is an Iwase manifold if and only if $M$ is an Iwase manifold. 

(2) If $\cat M^\bullet = \cat N^\bullet$, then $M\sharp N$ is an Iwase manifold if and only if both $M$ and $N$ are Iwase manifolds. 
\end{theorem}

We are now in position to prove the following theorem. 

\begin{theorem}\label{th:15} For any closed connected manifolds $M$ and $N$, we have $$\cat (M\sharp N)=\max\{\cat M, \cat N\}.$$
\end{theorem}
\begin{proof} The statement of Theorem~\ref{th:15} is well known in the case where the dimension of $M$ and $N$ is at most $2$. Suppose that the dimension of the manifolds under consideration is greater than $2$. 

  Suppose that $M\sharp N$ is an Iwase manifold. Without loss of generality we may assume that $\cat M^\bullet\ge \cat N^\bullet$. By Theorem~\ref{basis}, we have 
\[
    \cat (M\sharp N) = \cat (M^\bullet \vee N^\bullet ) =\max\{\cat M^\bullet, \cat N^\bullet\} =\max (\cat M, \cat N).    
\] 
Suppose now that $M\sharp N$ is not an Iwase manifold. Then
\[
    \cat (M\sharp N) = \cat (M^\bullet \vee N^\bullet ) + 1=\max\{\cat M^\bullet, \cat N^\bullet\} + 1 . 
\] 
Suppose first that  $\cat M^\bullet>\cat N^\bullet$. Then, by Theorem~\ref{th:2},  $M$ is not an Iwase manifold, and therefore
\[
  \cat (M\sharp N) = \max\{\cat M, \cat N\}. 
\]
Suppose now that $\cat M^\bullet =\cat N^\bullet$. Then, again by Theorem~\ref{basis},  either $M$ or $N$ is not an Iwase manifold, and again
\[
    \cat (M\sharp N) = \max\{\cat M, \cat N \} . 
\]

\end{proof} 

We conclude the paper by the following question
\begin{problem} Is the product of two Iwase manifolds Iwase?
\end{problem}

\end{document}